\documentclass{tran-l}
\usepackage[centertags]{amsmath}
\usepackage{amsfonts}
\usepackage{amssymb}
\usepackage{amsthm}
\usepackage{newlfont}
\usepackage{lscape}
\usepackage{graphicx}
\newlength{\defbaselineskip}
\newcommand{\setlinespacing}[1]%
          {\setlength{\baselineskip}{#1 \defbaselineskip}}

 \newcommand{\eps}{\varepsilon}
 \newcommand{\To}{\longrightarrow}
 \newcommand{\s}{\subseteq}
 \newcommand{\ep}{\epsilon}

 \newcommand{\gm}{\gamma}
 \newcommand{\Gm}{\Gamma}
 \newcommand{\Lm}{\lambda}

 \newcommand{\ten}{\otimes}

 \newcommand{\h}{\mathcal{H}}
 
 \newcommand{\rl}{\mathbb{R}}
 \newcommand{\Complex}{\mathbb{C}}

 \newcommand{\q}{\mathbb{Q}}

  \newcommand{\Z}{\mathbb{Z}}

 \newcommand{\od}{\mathcal{O}}

 \newcommand{\bth}{\begin{thm}}
 \newcommand{\brm}{\begin{rem}}
 \newcommand{\erm}{\end{rem}}
 \newcommand{\bcor}{\begin{cor}}
 \newcommand{\ecor}{\end{cor}}
 \newcommand{\ed}{\end{document}}
 \newcommand{\bmb}{\begin{mbs}}
 \newcommand{\emb}{\end{mbs}}
 
 \newcommand{\beq}{\begin{eqnarray*}}
 \newcommand{\eeq}{\end{eqnarray*}}

\theoremstyle{plain}
\newtheorem{thm}{Theorem}
\newtheorem{cor}[thm]{Corollary}
\newtheorem{lem}[thm]{Lemma}
\newtheorem{prop}[thm]{Proposition}

\theoremstyle{definition}
\newtheorem{defn}[thm]{Definition}
\newtheorem{rem}[thm]{Remark}

\numberwithin{equation}{section}

\begin{document}

\title[\noindent{\it Generators of Arithmetic Quaternion Groups and...}]{Generators of Arithmetic Quaternion Groups and a Diophantine Problem}

\author{Majid Jahangiri}

\begin{abstract}
 Let $p$ be a prime and $a$ a quadratic non-residue $\bmod~ p$.  Then the set of integral solutions of the diophantine equation $x_0^2 - ax_1^2 -px_2^2 + apx_3^2=1$ form a cocompact discrete subgroup $\Gamma_{p,a}\subset SL(2,\mathbb{R})$ and is commensurable with the group of units of an order in a quaternion algebra over $\mathbb{Q}$.  The problem addressed in this paper is an estimate for the traces of a set of generators for $\Gamma_{p,a}$.  Empirical results summarized in several tables show that the trace has significant and irregular fluctuations which is reminiscent of the behavior of the size of a generator for the solutions of Pell's equation.  The geometry and arithmetic of the group of units of an order in a quaternion algebra play a key role in the development of the code for the purpose of this paper.
\end{abstract}
\maketitle

\setlinespacing{1}


\section*{Introduction}
The solutions of the Pell's equation $x^2-dy^2=1$ form a discrete subgroup of an algebraic torus and is isomorphic to
$\mathbb{Z}$.  An estimate of the form $d^{\sqrt{d}}$ for the size of a generator of this
group is given by theorems of Schur and Siegel.
This estimate reflects the ``worst case" scenario and empirical results show that the actual size of the generator is {\it often} much smaller.
Here we consider the
diophantine equation $x_0^2-ax_1^2-px_2^2+apx_3^2=1$, where $a$
is a quadratic non-residue mod $p$.  This diophantine equation is analogous to Pell's equation in the sense that the integral solutions of
this equation form an arithmetic subgroup $\Gamma_{p,a}$ of $SL(2,\mathbb{R})$ and is commensurable with
the group of units of an order in a quaternion algebra over $\mathbb{Q}$. But unlike the case of Pell's equation, this
is a non-commutative (finitely generated) group and the question analogous to the size
of the generator for Pell's equation is an estimate for $x_0$ (or equivalently the traces) for a set of generators for this group. In this note we obtain
empirical estimates for these quantities for a set of $p$'s and
$a$'s, and according to the discriminant of the corresponding
quaternion algebras. These results are displayed in Figures 1-4 in \S 5 which confirm the rather irregular behavior of the size and number of generators as a function of the discriminant.

In order to find a set of generators for the group $\Gamma_{p,a}$ it is necessary to invoke the geometry of the fundamental domain for the action of the group on the Poincar\'e disc.  Once a fundamental domain is obtained a set of generators for the group can be exhibited and the main difficulty in the computer implementation is to make sure that a precise fundamental domain has been explicitly constructed.  The key idea is due to Hutchinson (in a 1907 paper [Ht]) who introduced a very useful ordering of the group elements and a construction for a fundamental domain which was later popularized by Ford and is usually known after the latter.  Hutchinson notes that his method is different from those in Fricke and Klein and lends itself much better to numerical implementation and explicit constructions.  Johansson noticed that by using only the geometry of a Ford fundamental domain (in the cocompact case) one can give an estimate for the size of generators for $\Gamma_{p,a}$ which is a cocompact subgroup in our case since $a$ is a quadratic nonresidue mod $p$.

The problem considered in this note has application to the explicit construction of the fundamental group of a compact orientable surface as a group of $2\times 2$ matrices which is often encountered by students of mathematics.  Poincar\'e's geometric construction does not provide a discrete subgroup of $PSL(2,\mathbb{R})$ sufficiently explicitly for displaying the desired matrix representation.  This article in fact grew out of an attempt to provide a method amenable to a computer implementation for explicitly exhibiting sets of generators for the fundamental groups of surfaces defined as arithmetic quotients of the upper half plane $\mathcal{H}$.

In sections 1 and 2 basic relevant facts about quaternion algebras
and their orders are recalled to fix notation and the geometry of the fundamental domain is discussed in section 3.  An account of the relevant aspects of the works Hutchinson and Johansson is given in \S 4 and applied to the case under consideration in \S 5.

It should be pointed that in a series of papers, J.H.H. Chalk [Ch-K \& Ch-A] obtained estimates for the size of generators for solutions of quadratic forms based on methods of analytic number theory.  However his estimates, when specialized to our case, are crude compared to those that one obtains by the method of Johansson.  The method of Corrales et al [Co] has similarities to that of Johansson but the work of the latter gives better effective bounds.  In [A-B] a method is described for explicitly exhibiting a set of generators for the group of units of Eichler orders in some special classes of quaternion algebras.

\section{Quaternion Algebras}
\begin{defn}
Let $K$ be a field with characteristic $\neq 2$. $H$ is a quaternion algebra over $K$ if it is a four dimensional algebra over $K$ with a basis $\{1,i,j,ij\}$ satisfying the product rules
  \begin{eqnarray*}
    i^2=a,~~~ j^2=b,~~~ ij=-ji,~~~~\;\;\;\; a,b\in K^*.
  \end{eqnarray*}
\end{defn}
Then it is well known that $H$ is a central simple $K$--algebra, and is denoted by $H=\left(\frac{a,b}{K}\right)$. There is an involution $h\mapsto \bar h$ taking $h=a+ib+jc+ijd$ to $\bar h=a-ib-jc-ijd$. The reduced trace and reduced norm of $h\in H$ are defined as:
\begin{eqnarray*}
  \hbox{{\rm tr}}(h)=h+\bar h , ~~~~~ n(h)=h\bar h
\end{eqnarray*}

By the Noether-Skolem Theorem for every separable quadratic
$K$--subalgebra $L$ of $H$ we have $H=L+L\omega$, for some
$\omega\in H$ with following properties:
\begin{align*}
\omega^2=&\theta\in K^*\\
\omega l=&\bar l\omega,~~ \hbox{ for all }l\in L.
\end{align*}
We refer to this representation as $H=\{L,\theta\}$.

If $E/K$ is any field extension, we have $\left(\frac{
a,b}{K}\right)\ten E=\left(\frac{a,b}{E}\right)$, and $\{L,\theta\}\ten E=\{L\ten E,\theta\}$. If $H\ten E\cong M(2,K)$ we say that $E$ is a splitting field of $H$. One can prove that a quadratic field extension $E/K$ splits a division quaternion algebra $H$ if and only if $E$ is $K$--isomorphic to a maximal subfield of $H$ containing $K$.

In this report, we consider the case where $K$ is a number field with the ring of integers $R$. For every place $v$ of $K$, let $K_v$ denote  the completion of $K$ at $v$.

\begin{defn}
  Let $H$ be a quaternion $K$--algebra. Then $H_v:=K_v\ten H$ is a quaternion $K_v$--algebra. If $H_v$ is a division algebra, we say that $H$ is ramified at $v$; otherwise we  say $H$ splits at $v$.
\end{defn}

\bth
There are exactly two quaternion algebras over $K_v$ up to isomorphism; one is the matrix algebra $M(2,K_v)$, and the other is $\mathbb{H}_v=\{L_{nr},\pi\}$, where $L_{nr}$ is the unique unramified quadratic extension of $K_v$, and $\pi$ is a uniformizer of $K_v$.
\end{thm}
For $H=\left(\frac{a,b}{K}\right)$ and $v$ a place of $K$, the Hasse invariant is defined as
\begin{eqnarray*}
  \ep(a,b)_v=\left\{
               \begin{array}{ll}
                 1, & \hbox{ if $H$ splits at $v$  ;}\\
                 -1, & \hbox{ if $H$ is ramified at $v$.}
               \end{array}
             \right.
\end{eqnarray*}
 If $K=\q$, the Hasse invariant coincides with the Hilbert symbol $(a,b)_v$:
\begin{eqnarray*}
  (a,b)_v=\left\{
            \begin{array}{ll}
              1, & \hbox{ if $X^2-aY^2-bZ^2$ is isotropic over $K_v$;} \\
              -1, & \hbox{ otherwise.}
            \end{array}
          \right.
\end{eqnarray*}
For computing $(a,b)_p$, the following properties of the Hilbert symbol are useful:\\
\begin{enumerate}
  \item $(a,bc)_p=(a,b)_p(a,c)_p$.
  \item if $p$ is a real place, then $(a,b)_p=-1$ iff  $a<0,~b<0$. In the case $(a,b)_p=-1$ we say that $H$ is definite, and otherwise $H$ is indefinite.
  \item If $p\neq 2$ and $a$ and $b$ are co-prime, then $(a,b)_p =\left\{
                                           \begin{array}{ll}
                                             1, & \hbox{if } a,b\in R^*_p~\\&(\hbox{or}~ p\nmid a,p\nmid b); \\
                                             \left(\frac{a}{p}\right), & \hbox{if }p\nmid a, p\mid b.
                                           \end{array}
                                         \right.$
  \item The number of places that is ramified in $H$ is finite and $\prod_p(a,b)_p=1,\,\,$ $\,\forall a,b\in K^*$, so that $(a,b)_2=\prod_{p\neq 2} (a,b)_p$
\end{enumerate}
\bth
\begin{enumerate}
 \item Two quaternion $K$--algebras are isomorphic if and only if they are ramified at the same places.
  \item Given an even number of non-complex places of $K$, there exist a quaternion $K$--algebra that ramifies exactly at these places.
\end{enumerate}
\end{thm}
\begin{proof} See [M-R] Theorem 7.3.6.\end{proof}

\begin{defn}
The (reduced) discriminant $d_H$ of a quaternion algebra $H/K$ is the product of prime ideals of $R$ that ramify in $H$. We will refer to the set of ramified primes by Ram$(H)$.
\end{defn}

By the Hilbert Reciprocity Law Ram$(H)$ is finite of even cardinality.  In the case of $K=\q$, $d_H$ is an integer, which if $H$ is definite, is a product of an odd number of distinct primes, and if $H$ is indefinite, is a product of an even number of distinct primes. In this paper, we mainly concerned with the indefinite cases.

Let $H=\left(\frac{a,b}{\q}\right)$ be an indefinite quaternion algebra, with $a>0$. We fix the embedding $\Phi:H\To M(2,\q\sqrt{a})$ with
\begin{eqnarray}\label{3}
  \Phi(x_0+x_1i+x_2j+x_3ij)=\left(
                              \begin{array}{cc}
                                x_0+x_1\sqrt{a} & x_2+x_3\sqrt{a} \\
                                b(x_2-x_3\sqrt{a}) & x_0-x_1\sqrt{a} \\
                              \end{array}
                            \right)
\end{eqnarray}
For an indefinite quaternion algebra $\;H\;$ over $\;\mathbb{Q}\;$  denote the discriminant by $d_H=d=p_1p_2...p_{2m}$.  By the Chinese Reminder Theorem, we can choose an integer $a$ such that $\left(\frac{a}{p_i}\right)=-1$ for all $p_i>2$. Furthermore by Dirichlet's Theorem on primes in an arithmetic progression one can assume $a$ is a prime $p$ such that $p\equiv 5$ (mod 8). For such a $p$, the Hilbert symbol:
\beq
(p,d)_{p_i}=(p,p_i)_{p_i}=\left(\frac{p}{p_i}\right)=-1, \hbox{ for all }p_i>2.
\eeq
Now $p\equiv5 (\bmod ~8)$ implies $\left(\frac{2}{p}\right)=-1$. Since $p\equiv1 (\bmod 4)$ we have $\left(\frac{p}{p_i}\right)\left(\frac{p_i}{p}\right)=~1$, and therefore
\beq
(p,d)_p=\prod_{i=1}^{2m}\left(\frac{p_i}{p}\right)=1.
\eeq
Hence $\left(\frac{p,d}{\q}\right)$ ramifies at all primes dividing $d$, and is  unramified at $p$, and at all other odd primes. So $\left(\frac{p,d}{\q}\right)$ has the same discriminant as $H$, and $H\cong \left(\frac{p,d}{\q}\right)\cong \left(\frac{p,-d}{\q}\right)$. We will make use of this representation of an indefinite rational quaternion algebras.

\section{Quaternion Orders}

\begin{defn}
  An $R$--order $\od$ in a quaternion algebra $H/K$ is a complete $R$--lattice in $H$ which is also a ring with unity.
\end{defn}

By an {\it integer} of $H$ we mean an element $a\in H$ such that $\hbox{{\rm tr}}(a)$ and $n(a)$ are in $R$. Each element of an order is an integer in $H$.  For an order $\od$, we define the dual lattice as
\begin{eqnarray*}
  \od^{\#}=\{x\in H ;\,\, \hbox{{\rm tr}}(x\od)\s R\}
\end{eqnarray*}
which is a complete $R$--lattice in $H$.

\begin{defn}
  Let $\od$ be an order of $H$. The reduced discriminant of $\od$, which denotes by $d(\od)$, is the ideal $n(\od^{\#^{-1}})\s R$.
\end{defn}

One can prove that if $\od=R[x_1,...,x_4]$, then $n(\od^{\#^{-1}})^2=R\det[ \hbox{{\rm tr}}(x_ix_j)]_{i,j},$ $1\leq i,j \leq 4$. So for two orders $\od'\s \od$, the discriminants satisfies
\begin{eqnarray*}
  d(\od')=[\od:\od']d(\od)
\end{eqnarray*}
and $d(\od)=d(\od')$ implies that $\od=\od'$. We say that an order is maximal, if it is not properly contained in any other order. An Eichler order is the intersection of two maximal orders.

\begin{prop}
  Let $H=\{L,\theta\}$ be a quaternion algebra over a local field $K_v$, and $\mathfrak{p}$ be the maximal ideal of $R_v$. Let $\od$ be a maximal order in $H$. If $H\cong M(2,K_v)$, then $\od$ is conjugate to $M(2,R_v)$, and $d(\od)=R_v$. Otherwise $H=L+L\omega$ is a division algebra, $\od=R_L+R_L\omega$ is unique, and $d(\od)=\mathfrak{p}$.
\end{prop}

Let $K$ be a number field. An  $R$--order $\od$ in $H/K$ is a maximal order if and only if $\od_p:=R_p\ten \od$ is a maximal  $R_p$--order for each place $p$ of $K$. This follows from:
\begin{eqnarray*}
\od=H \cap \od_2\cap\od_3\cap ...\cap\od_p\cap...\;\cdot
\end{eqnarray*}
From this, and the relation $d(\od_p)=d(\od)_p$, we conclude that
\beq \od \hbox{ is a maximal order } \Leftrightarrow d(\od)=
\prod_{p\in Ram(H)} p=d(H). \eeq This gives a criterion for the
characterization of maximal orders by their discriminants.

\section{Arithmetic Groups}
Let $\h=\{z\in \Complex \,\,;\,\Im(z)>0\}$ be the Poincar\'{e} half-plane with the hyperbolic metric
\beq
ds^2=\frac{dx^2+dy^2}{y^2}.
\eeq
 Let $\od$ be an order in $H=\left(\frac{a,b}{\q}\right)$, and
\beq
\od^1=\{x\in \od;\,\,n(x)=1\}
\eeq
be the group of proper units of it. Under $\Phi$ (see \ref{3}) the image of $\od$ lies in $PSL(2,\rl)$. One can show that this group is a discrete subgroup of $PSL(2,\rl)$ [M-R, Chapter 8]. $PSL(2,\rl)$ acts on $\h$ by M\"{o}bius transformations
\beq
g\cdot z=\frac{az+b}{cz+d}, \hbox{ for } g=\left(
                                             \begin{array}{cc}
                                               a & b \\
                                               c & d \\
                                             \end{array}
                                           \right)\in PSL(2,\rl)\cdot
\eeq
Then $\Gm:=\Phi(\od^1)$ also acts on $\h$ properly discontinuous and freely, i.e. has no fixed points. Since we can transform $\h$ to the unit disc $\mathcal{U}$ by the action of matrix $\psi=\left(
   \begin{array}{cc}
     i & 1 \\
     1 & i \\
   \end{array}
 \right)$, which sends the $i$ to the origin, the conjugation of $\Gm$ by $\psi$ gives an embedding $\bar \Gm$ of $\od^1$ in $SU(1,1)$.
\begin{defn}
  Given $\gm=\left(
               \begin{array}{cc}
                 a & b \\
                 c & d \\
               \end{array}
             \right)
\in SL(2,\Complex)$ such that $c\neq0$, the circle $C_\gm :=\{z\in \Complex\,;\,\,|cz+d|=1\}$ is called the isometric circle of $\gm$. We denote by $r_\gm$ and $o_\gm $ its radius  and center, respectively.
\end{defn}
The isometric circle is characterized by the property that $\gm$ acts on it as a Euclidean isometry. By a straightforward computation, the center and the radius of $C_\gm$ are the numbers $o_\gm=-d/c$ and $r_\gm=1/|c|$, respectively.

The action of $PSL(2,\rl)$ extends to $\h\cup\rl\cup\{\infty\}$, and we have three types of elements in the group action, depending on their fixed points. Fixed points are determined by the solutions of
\beq
f(z)=cz^2+(d-a)z-b=0.
\eeq
We have three cases according to the sign of the discriminant $\delta(f)=(a+d)^2-4$.
\begin{defn} Assume $\gm=\left(
                           \begin{array}{cc}
                             a & d \\
                             b & c \\
                           \end{array}
                         \right)$ defines a transformation different from $\pm$Id. $\gm$ is
   \begin{itemize}
       \item hyperbolic, if it has two distinct fixed points in $\,\rl\,\cup\,\infty\,\,\,$, equivalently if $\,\,|\hbox{{\rm tr}}(\gm)|>2$;
       \item  elliptic, $\,\,$if it has two complex conjugate fixed points $\,\,$, equivalently if $\,\,|\hbox{{\rm tr}}(\gm)|\,\,<2$;
       \item parabolic, $\,\,$ if it has $\,$ a $\,$ unique fixed point in $\,\,\rl\,\cup\,\infty\,\,$,  equivalently if $|\hbox{{\rm tr}}(\gm)|=2$.
     \end{itemize}
\end{defn}

By using the Hasse-Minkowski Principle on quadratic forms, one can prove that the quaternion groups, i.e. groups arising from the proper units of an order in an indefinite quaternion algebra have no parabolic elements. Elliptic and hyperbolic elements are diagonalizable as    $ \left(
  \begin{array}{cc}
    \Lm & 0 \\
    0 & \Lm^{-1} \\
  \end{array}
\right)$. For elliptic elements, $\Lm^2=e^{i \theta}$, and they are torsion elements, and for hyperbolic elements $\gm\neq1$ is real. In the case of quaternion groups elliptic elements are only of orders 4 or 6,
\begin{defn}\label{6}
A connected closed polygon $\mathcal{D}\s\h\cup\rl\cup\{\infty\}$ is a fundamental domain for the action of $\Gm$ on $\h$ if any two points $z,z'$ in the interior of $\mathcal{D}$ are not $\Gm$--equivalent, i.e. there is not any $\gm\in \Gm$ such that $z=\gm(z')$, and if each point in $\h$ is $\Gm$--equivalent to some point in $\mathcal{D}$.
\end{defn}

Ford describes a method to construct a fundamental domain for the action of any discrete subgroup of $PSL(2,\rl)$ on the unit disc by using their isometric circles [F1]. His method is applicable to quaternion groups of interest in this paper and the resulting fundamental domains have finitely many sides. In Ford's method the region which is exterior to all isometric circles of elements of $\Gm$ constitute a fundamental domain of the action of $\bar\Gm$ on the unit disc. The set of elements $\gm$ corresponding to the sides of this region give us a set of generators of $\Gm$; however this set is not necessarily  minimal.

By attaching a Dedekind zeta function to a maximal order $\od$ in an indefinite quaternion algebra, and relating the residue of its pole at $s=1$ to the hyperbolic area of the fundamental domain of the action of the corresponding unite group $\Gm$ on $\h$, Eichler [E2] showed that
\begin{eqnarray}\label{9}
\hbox{{\rm Vol}}(\Gm\backslash\h)=\frac{\pi}{3}\prod_{p\mid d(\od)}(p-1)
\end{eqnarray}
Vign\'{e}ras [V] recovered this result by using local-global Tamagawa measures and adelic techniques, and computed the area of the fundamental domain for a maximal order with respect to the Poincar\'e arithmetic metric $d\bar s:=ds/2\pi$. With this normalization the volume is the rational number $\frac{1}{6}\prod_{p|d(\od)}(p-1)=\frac{1}{6}\phi(d(\od))$. (In [E2] and [V] these computations are considered for the wider class of Eichler orders of level $N$)

\section{Fundamental Domains And Generators}
For finding the Ford fundamental domain and a set of generators, one computes isometric circles for elements $\bar\Gm$. For a computer implementation of this, it is necessary to order elements of $\bar\Gm$. Then isometric circles are computed one at a time according to this order until a closed domain $\mathfrak D$ is obtained. This $\mathfrak D$ is not necessarily a fundamental domain for $\bar\Gm\backslash \mathcal{U}$, and the actual fundamental domain may be smaller since the isometric circles of the other elements of $\bar\Gm$ may cut across $\mathfrak D$. To address these problems we proceed as follows.

An upper bound for the generators is obtained in [Ch2] and [Ch-K]. Defining the norm of $\gm\in\Gm\s SLS(2,\rl)$ as
\beq
||\gm||=a^2+b^2+c^2+d^2\,\,\,\,\hbox{ for }\,\,\,\gm=\left(
                                                       \begin{array}{cc}
                                                         a & b \\
                                                         c & d \\
                                                       \end{array}
                                                     \right),
\eeq
and assuming that $\Gm\backslash\h$ is  compact, Chalk proved that there exists a set of generators $\{A_1,...,A_N\}$ for $\Gm$ with the following properties:

\begin{enumerate}
  \item $||A_1||\leq||A_2||\leq...\leq||A_N||$
  \item $||A_1||\leq c_1N$
  \item $||A_{i+1}||\leq c_2N||A_i||^5\,\,\,\,(i=1,...,N-1)$.
\end{enumerate}
He also computes the explicit constants $c_1$ and $c_2$ in the case of maximal orders of an indefinite quaternion algebra over $\q$ [Ch-K], and obtained the  the following bounds for the generators of the group of proper units of maximal orders in a quaternion algebra with the representation $H=\left(\frac{p,d_H}{\q}\right)$:
\begin{align*}
&N\leq 6+\phi(d_H),\\
&||A_1||< 2+4\pi^{-1}N,\\
&||A_{i+1}||< \,\,\frac9 {64} \,\cdot\,\frac{p^2}{d_H^2}\,\,N\,\,||A_i||^5.
\end{align*}
The bound for $||A_1||$ is improved in  [Ch2] and [Ch-A]. The exponent $5$ for $||A_{i+1}||$ in terms of $||A_i||$ yields a crude and practically inapplicable estimate.\\

For a faster implementation of the construction of the true fundamental region $\mathcal{D}$ we use two methods described by Hutchinson [Ht] and Corrales et.al.[Co]. Let $F(z)=1-z\bar z$, then $F>0$ represents the interior of the unit disc. One observes that if $|z|<|z'|$, then $|F(z)|>|F(z')|$. Therefore among the points of isometric circle $C_\gm$ associated to any $\gm =\left( \begin{array}{cc} a & \bar c \\ c & \bar a \\\end{array} \right)\in\bar\Gm$, $F$ assumes a maximum on the nearest point to the origin which we denote by $z_\gm$. The boundary of $\mathcal D$ consist of arcs of isometric circles that are closest to the origin and $F$ assumes its maxima on these circles. We have:
\begin{eqnarray}
F(z_\gm)=\frac 2{|a|+1}.
\end{eqnarray}
So the boundary circles for $\mathcal{D}$ occur among those $\gm$'s determined by the smallest values of $|a|$. Hence, we start with the small values for $|a|$, and compute the required norm for $c$ by $|c|=\sqrt{|a|^2-1}$, and the above descriptions assure us that the first time we attain a closed region, it will be a fundamental domain for $\bar\Gm$.

The method of Corrales et.al. is based on the assumption that $\Gm$ is cocompact. If $c\in \h$ is not a fixed point of $\Gm$, and $\rho $ denotes the hyperbolic distance, define
\beq
D_g(c):=\{x\in\h\,\,;\,\, \rho(x,c)\leq \rho(x,g(c))\},\,\,\,\,\,\,\,g\in \Gm
\eeq
to be the half plane containing $c$. Then
\beq
D_\Gm(c):=\bigcap_{1\neq g\in \Gm} D_g(c)
\eeq
is known as a Dirichlet fundamental domain of $\Gm\backslash\h$. Now for a positive real number $r$ set
\begin{eqnarray}\label{4}
D_r(c)=\bigcap \{D_g(c)\,\,|\,\,1\neq g\in \Gm \hbox{ with }g(c)\in \bar B(c,r)\}
\end{eqnarray}
where $\bar B(c,r)=\{x\in \h\,\,;\,\,\rho (x,c)\leq r\}$. Since $D_\Gm(c)$ is compact, there exist a positive number $R$ such that $D_\Gm(c)= D_R(c)$. For finding such $R$ we let $k_1,k_2,k_3,...$ be an increasing sequence of positive numbers (goes to infinity), and compute $D_{k_i}$'s until $D_{k_n}$ is compact for some $k_n$. By the relation [B, Theorem 4.2.1]:
\begin{eqnarray}\label{7}
||g||=2\cosh\rho(i,g(i))
\end{eqnarray}
the computation of $D_{k_i}$'s is reduced to those elements $g\in \Gm$ for which $||g||<k_i$. Now set $R=2(\max(\{k_n/2\}\cup\{\rho(c,x)\,\,;\,\,x\in D_{k_n}(c)\}))$. For this $R$ we see that $\mathcal D\subseteq D_{k_n}(c)\s \bar B(c,R/2)$. Also for any $g\in\Gm$ which $\rho(c,g(c))>R$, we have $\bar B(c,R/2)\s D_g(c)$, because if there is some $x\in \bar B(c,R/2)\backslash D_g(c)$, then $\rho( x,g(c)) < \rho(x,c) \leq R/2$, and so $R< \rho(c,g(c))\leq \rho(c,x)+\rho(x,g(c)<R$, a contradiction. It follows that for the computation of the fundamental domain and the determination of generators, it suffices to consider elements $g\in\Gm$ with entries of $g$ bounded by $R$.\\

The method [Co] is computationally intensive since it requires the calculation of all the domains $D_{k_i}$. By a good choice of $k_1$ the computational burden can be reduced. This method can be translate to the Ford fundamental domain, if instead of (\ref{7}) we use the relation:
\begin{eqnarray}\label{5}
r_{\bar\gm}= \frac{1}{\sqrt{||\gm||-2}},~~~~~\hbox{ for }\bar\gm\in\bar\Gm.
\end{eqnarray}
Since $\mathcal D$ is compact, then there is a disc
\beq
\mathcal U_\eps=\{z\in \Complex \,\,;\,\,|z|\,<\,1-\eps\,\},\; 0<\eps<1,
\eeq
such  that $\mathcal D\s \mathcal U_\eps$, or equivalently
for a set of generators $\{\bar\gm_i\}$ we have $r_{\bar\gm_i}>\varepsilon$. So any good $\eps$ which estimates $r_{\bar\gm_i}$'s, gives an upper bound for entries of $\gm_i$'s. Since we want that $\mathcal D\s \mathcal U_\eps$, a crude bound can be assumed with respect to the area of $\mathcal{D}$ as:
\begin{eqnarray}\label{2}
\eps<\frac{1}{k}\left(1-\sqrt{\frac{\hbox{{\rm Vol}}(\mathcal{D})}{2+\hbox{{\rm Vol}}(\mathcal{D})}} \,\, \right)
\end{eqnarray}
where $k$ is larger than 2. In order to apply this method, it is
necessary to know the area of $\mathcal{D}$. In the next section we introduce a method for computing {\rm Vol}$(\mathcal D)$ for quaternion algebras over $\mathbb{Q}$.

\section{The Main Illustration}

Here we want to work in a special class of examples, and apply the above methods to obtain a bound on the trace and norm of a set of generators. We consider integer solutions of the quadratic Diophantine equation
\beq
x_0^2-ax_1^2-px_2^2+apx_3^2=1
\eeq
where $p$ is an odd prime, and $a<p$ is a quadratic non-residue mod $p$. Let
\beq
\Gm_{p,a}=\left\{\left(
             \begin{array}{cc}
               x_0+x_1\sqrt a & \sqrt p(x_2+x_3\sqrt a) \\
               \sqrt p(x_2-x_3\sqrt a) & x_0-x_1\sqrt a \\
             \end{array}
           \right)\,\,|\,\,\begin{array}{c}
                             x_0^2-ax_1^2-px_2^2+apx_3^2=1 \\
                             x_i\in \Z\,\,\,\forall i=0,...,3
                           \end{array}
           \right\}.
\eeq
Then $\Gm_{p,a}$ is the group of units of the \textit{canonical} order $\od_{p,a}=\Z[1,i,j,ij]$ in the quaternion algebra $H=\left(\frac{p,a}\q\right)$, and $x_0^2-ax_1^2-px_2^2+apx_3^2$ is the norm on $H$.
\begin{lem}
  For $p\equiv1$ (mod $4$) the group $\Gm_{p,a}$ is torsion free.
\end{lem}

\begin{proof}
Let Id$\neq\gm\in \Gm_{p,a}$ be a torsion element. Then its eigenvalues are $e^{\pm i\theta}$ and the characteristic equation is $\Lm^2 -(2\cos\theta) \Lm+1=0$. Therefore $x_0=0,\pm 1$, and $\gm$ is conjugate to a rotation matrix. In the latter case, $x_0=\pm1$, $\gm$ is conjugate to $\pm $Id, and so $\gm=\pm $Id. In the former case, $x_0=0$, the norm form becomes $1+ax_1^2+px_2^2= apx_3^2$. Reducing mod $p$ gives us
\beq
1+ax_1^2\equiv 0 \hbox{ mod }p
\eeq
which is not possible, since $-1$ is a quadratic residue and $a$ is not.
\end{proof}
\begin{rem}
The trace of an element of $\Gm_{p,a}$ is an even integer. Since an element of
finite order is conjugate to a rotation matrix, every non-trivial element of finite order in $\Gm_{p,a}$ necessarily has trace zero. Let $\Gm^0_{p,a}$ be the subgroup of $\Gm_{p,a}$, consisting of elements with $x_2\equiv 0 \hbox{ mod }a$. It is trivial that such matrices form a subgroup of finite index in $\Gm_{p,a}$. For $\gm\in\Gm^0_{p,a}$ we have $x_0\equiv 1$ mod $a$, and therefore $\Gm^0_{p,a}$ is torsion free.
\end{rem}
We apply Hutchinson's method to the given $\Gm_{p,a}$. First we transform $\Gm_{p,a}$ to the group $\bar\Gm_{p,a}$ which acts on the unit disc:
\beq
\gm=x_0+ix_1+jx_2+ijx_3\mapsto \bar\gm=\left(
             \begin{array}{cc}
               x_0+i x_3\sqrt{ap} & x_1\sqrt a-i x_2\sqrt p \\
               x_1\sqrt a+i x_2\sqrt p & x_0-i x_3\sqrt{ap} \\
             \end{array}
           \right)\in SU(1,1)
\eeq
Since the Hutchinson's method works step by step to find the fundamental domain, we wrote a computer program which generates $\bar\gm\in \bar\Gm_{p,a}$, and stops when a compact domain was obtained. The following four diagrams are the bounds founded by the program experimentally for $x_0$'s, and number of generators, for some pairs $(p,a)$. It is clear that the bounds for $x_1,x_2$ and $x_3$ are comparable to those for $x_0$.\\

\begin{figure}[p]
  \includegraphics[width=12.8cm]{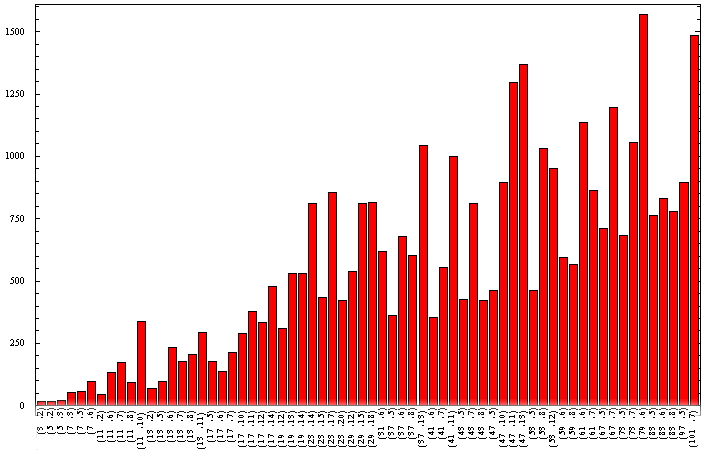}\\
  \caption{Number of generators plotted against the pairs $(p,a)$}
\end{figure}
\begin{figure}[p]
  \includegraphics[width=12.8cm]{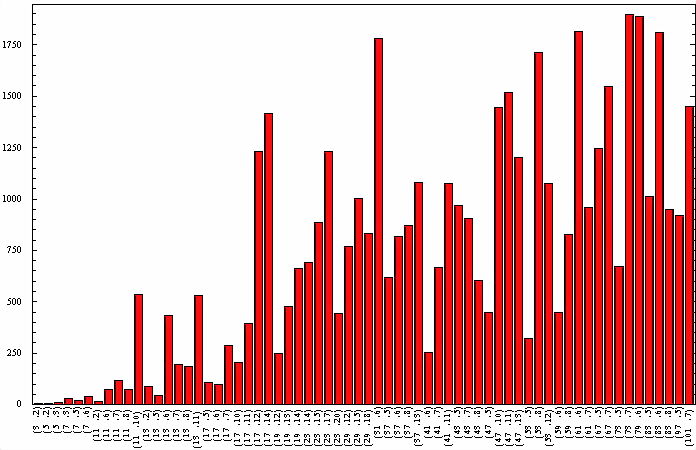}\\
  \caption{Bounds for $x_0$'s plotted against the pairs $(p,a)$}
\end{figure}
\begin{figure}[p]
  \includegraphics[width=12cm]{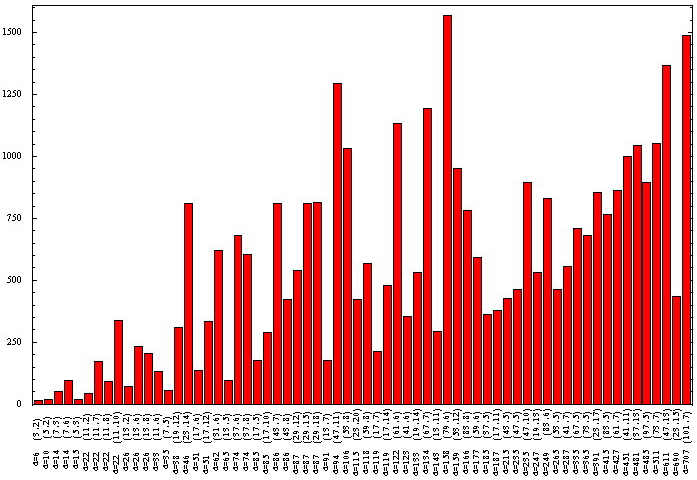}\\
  \caption{Number of generators ordered by discriminant $d_H$}
\end{figure}
\begin{figure}[p]
  \includegraphics[width=12cm]{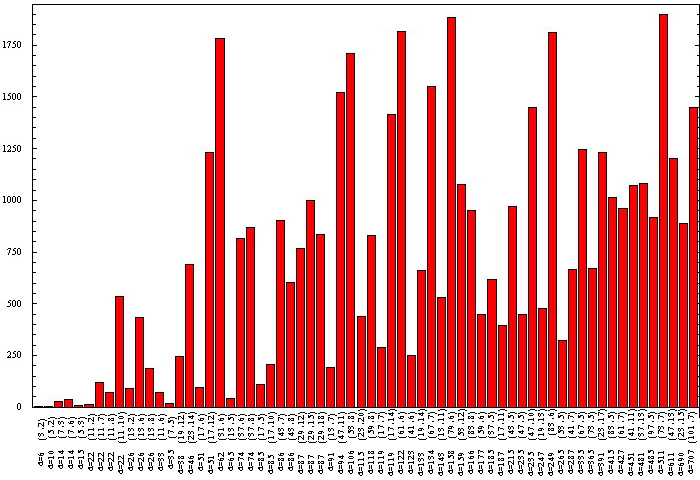}\\
  \caption{Bounds for $x_0$'s ordered by discriminant $d_H$}
\end{figure}

For these quaternion algebras, we can also apply the method of Johansson to find good bounds for $x_i$'s. In order to do this we compute the area of the fundamental domain of the group of proper units $\mathcal M^1$ of a maximal or Eichler order $\mathcal{M}\supset \od_{p,a}$ on $\mathcal{U}$, and then apply the formula [V]:
\beq
\hbox{{\rm Vol}}(\bar\Gm_{p,a}\backslash\mathcal{U})=[M^1:\Gm_{p,a}]\,\,\hbox{{\rm Vol}}(\bar{\mathcal{M}}^1\backslash\mathcal{U})
\eeq
to determine the area of the $\bar\Gm_{p,a}$. Then (\ref{5}) and (\ref{2}) give bounds for the entries of $\gm$'s. The main problem here, is to compute the index $[\mathcal M^1:\Gm_{p,a}]$. There are two distinct methods for this purpose. One is given in [Ch1]. This method is based on the reducing problem to the theorems about quadratic forms on finite fields by using [Ch1, Theorem 2]
\beq
[\mathcal M^1:\Gm_{p,a}]=\prod_{q\,\,\hbox{\tiny{\rm prime}},\,\,q|2m}[(\mathcal M^1)_{q^n}:(\Gm_{p,a})_{q^n}]
\eeq
where $m $ is an integer such that $m\mathcal M\s\od_{p,a}$, $n=v_q(m)$ is the $q$--adic valuation of $m$ if $q$ is odd, and $n=\max(8,v_2(m))$ when $q=2$, and
\beq
(\Gm_{p,a})_{q^n}=\{u\in\od_{{p,a}_q}\,\,;\,\,n(u)\equiv1 (\hbox{ mod } q^n)\},\\
(\mathcal M^1)_{q^n}=\{v\in \mathcal M_q\,\,;\,\,n(v)\equiv1 (\hbox{ mod }q^n)\}.
\eeq

For the computation of index, we combine the method described in [J3], the relation [V, IV.1.7]:
\beq
[\od^1:\od'^1]=\prod_p[\od^1_p:\od'^1_p] ~~\;\;~~~~~~\hbox{ which } ~~~\od'\s\od,
\eeq
where $\od_p=\od\otimes R_p$, and the Eichler invariant in the localizations of the orders. The Eichler invariant for $\od_p\neq M(2,R_p)$ is defined as
\beq
e(\od_p)=\left\{
           \begin{array}{ll}
              1 & \hbox{ if $\od_p/J(\od_p)\cong k_p\oplus k_p$;} \\
             0 & \hbox{ if $\od_p/J(\od_p)\cong k_p$;} \\
             -1 & \hbox{ if $\od_p/J(\od_p)$ is a quadratic extension of } k_p,
           \end{array}
         \right.
\eeq
where $k_p$ is the residue field of $K_p$ and $J(\od_p)$ is the Jacobson radical of $\od_p$. These are the only possibilities for $\od_p/J(\od_p)$, because the dimension of semisimple algebra $\od_p/J(\od_p)$ over $k_p$ is at most $4$. In fact, it cannot be a cubic or quartic field over $k_p$ because
this would be a simple separable extension of $k_p$, and it would follow from
Hensels lemma that this can be lifted to a field contained in the algebra over
the $K_p$ which is impossible. Also it cannot be a quadratic field
times $k_p$ because idempotents can also be lifted, and then the algebra over $K_p$ would contain a corresponding subalgebra which is again
impossible by the Double Commutator Theorem.  So the only possibilities are those listed above.

\begin{lem}
  In the rational quaternion algebra $H_p=H\otimes \q_p$, let two orders $\od_p\subsetneq \mathcal M_p$ be such that $\mathcal M_p$ is maximal in $H_p$ and $d(\od_p)=p^n$. Then
  \beq
  [\mathcal M_p^1\,:\,\od_p^1]=[\mathcal M_p^*\,:\,\od_p^*] / [\Z_p^*\,:\,n(\od_p^*)]
  \eeq
  and
  \beq
  [\mathcal M_p^*\,:\,\od_p^*]=\left\{
                                 \begin{array}{ll}
                                   p^{n-1}(p^2-1)(p-e(\od_p))^{-1} & \hbox{ if }H_p\cong M(2,\q_p),  \\
                                   p^{n-1}(p+1)(p-e(\od_p))^{-1} & \hbox{ if } H_p\cong \mathbb H_p.
                                 \end{array}
                               \right.\\
  \eeq
\end{lem}
\begin{proof} See [Kr] Theorem 1. \end{proof}

For an order $\od$ with the dual lattice $\od^\#$ and the dual basis $\{1,b_1,b_2,b_3\}$ we assign the quadratic form
\beq
  f_\od(X_1,X_2,X_3):=d(\od)\cdot n(X_1b_1+X_2b_2+X_3b_3).
\eeq
\begin{prop}
  Let $\od$ be an order in the rational quaternion algebra $H$, and $p$ be an odd prime.  If $f_\od=X_1^2+up^rX_2^2+wp^sX_3^2$ with $r\leq s$ and $(uw,p)=1$, then
\beq
 e(\od_p)=1\;\;\Longleftrightarrow\;\;r=0,\;s\geq1,\;\left(\frac{-u}p\right)=1,\\
 e(\od_p)=-1\;\;\Longleftrightarrow\;\;r=0,\;s\geq1,\;\left(\frac{-u}p\right)=-1.
\eeq
In this case $e(\od_2)=0$.
\end{prop}
\begin{proof} See Theorem 4.3 and Propositions 5.8 \& 5.9 in [J3], and Satz 10 in [E1]. \end{proof}

Finally Johansson obtained the following formula for the volume of a fundamental domain in the case of rational indefinite quaternion algebras [J1, Prop.5.10]:
\begin{align*}
\hbox{{\rm Vol}}(\od^1\backslash\h)=&\frac\pi 3\prod_{q|d_H}(q-1)\prod_{q|d(\od)}[\mathcal{M}^1_q:\od^1_q]\\=&\frac \pi 3 d(\od) \prod_{q|d(\od)}[\mathbb{Z}_q^*:N(\od_q^*)]^{-1}\frac{q^2-1} {q(q-e(\od_q))}
\end{align*}
which $\mathcal{M}$ is the maximal order of $H$ and $q$ is prime. In the case of $\od^1_{p,a}=\Gamma_{p,a}$ we have:
\beq
 f_{\od_{p,a}}=-X_1^2+aX_2^2+pX_3^2.
\eeq
Also for $q$ an odd prime $[\mathbb{Z}_q^*:N(\od_q^*)]=1$ since $p$ is a prime and $a<p$ (cf. [J1] Lemma 5.3).
\begin{figure}[]
  \includegraphics[width=13.7cm]{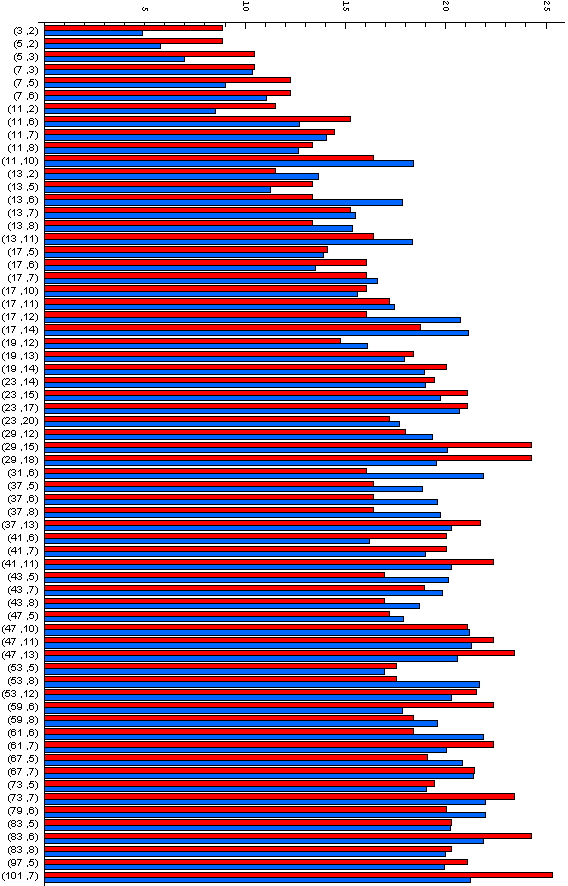}\\
  \caption{A comparison between the results of Hutchinson's method (the blue bars) and Johansson's bounds (the red bars), in logarithmic scale, about generators norm.}
\end{figure} For $q=2$ the index depends only on the power of $2$ in the decomposition of $a$.  In fact, $[\mathbb{Z}_2^*:N(\od_2^*)]=2$ or $1$ according as $4|a$ or $4\nmid a$, respectively. 

In order to evaluate the efficacy of Johansson's method, we compared the results with the exact answers one obtains from Hutchinson's method.  The results are shown in Figure 5 (it is plotted in logarithmic scale).  Note that Johansson's method can over-estimate or under-estimate by factors as large as 55 and 27 respectively, in the range of computed values.  The average discrepancy with the exact values is almost a factor of 4.

\section{Open Problems}
In this section we briefly describe three related open problems.

\begin{enumerate}
\item  One can similarly consider families of cocompact arithmetically defined discrete subgroup of $SL(2,\mathbb{C})$ acting properly discontinuously on the upper half space and ask analogous questions about the size of a set of generators for such groups.
\item The results in this note have been empirical and we do not know of an analogue of theorems of Schur and Siegel for units of quaternion algebras which are definite over $\mathbb{Q}$.    It is possible that analytical methods based on $L$--functions and representation theory can be used to obtain better estimates for the size of a set of generators.
\item Here we only considered discrete groups commensurable with the group of units of a quaternion algebra definite over $\mathbb{Q}$.  One may ask similar questions about subgroups of finite index in groups of units of orders in quaternion algebras definite over a totally real number field and acting properly discontinuously on a product of upper half planes.  It is an interesting problem whether the construction of [A-B] for certain Eichler orders can be extended to this case.
\end{enumerate}

\vskip .5 cm

\noindent {\bf Acknowledgement} - The author would like to thank Professors P. Bayer, Stefan Lemurell and Kleinert for very helpful communications, and especially Professor M. Shahshahani for suggesting the problem and consistent encouragement.

\vskip .8 cm
\setlinespacing{1.3}
\noindent {\bf Referenes}

\noindent [A-B] \textsc{Alsina}, M. and \textsc{Bayer}, P. ; {\it Quaternion Orders, Quadratic Forms, And Shimura Curves}, CRM Monograph Series(Vol. 22), 2004.

\noindent [B] \textsc{Beardon}, A.F. ; {\it The Geometry of Discrete Groups}, Springer-Verlag, 1983.

\noindent[Ch1] \textsc{Chalk}, J.H.H. ; Units of Indefinite Quaternion Algebras, {\it Proc. Royal Soc. Edingburgh}, vol. 87A(1980), 111-126.

\noindent [Ch2] ----------- ; Inequalities for Pell Equations and Fuchsian Groups, {\it Discrete groups and geometry (Birmingam, 1991)} , 26-35, London Math Soc, Lecture Note Ser. (173), {\it Cambridge Univ. Press}, 1992.

\noindent [Ch-A] ----------- and \textsc{Ashton},R.J. ; On The Representation of Integers by Indefinite Diagonal Quadratic Forms, {\it C.R. Math. Rep. Acad. Sci. Canada}, vol. 16(1994), no 1, 23-24.

\noindent [Ch-K] ----------- and \textsc{Kelly}, B.G.A. ; Generating Sets for Fuchsian Groups, {\it Proc Royal Soc. Edingburgh}, vol. 72A(1973), 317-326.

\noindent [Co] \textsc{Corrales}, Capi, et.al. ; Presentation of the unit group of an order in a non-split quaternion algebras, {\it Adv. Math.}, vol. 186(2004), no. 2, 498-524.

\noindent [E1] \textsc{Eichler}, M. ; Untersuchungen in der Zahlentheorie der rationalen Quaternionalgebren, {\it J. Reine Angew. Math.}, vol. 174(1936), 129-159.

\noindent [E2] ----------- ; {\it Lectures on Modular
Correspondences}, Tata Institute, Bombay (1955-6) (re-issued 1965).

\noindent [F1] \textsc{Ford}, L.R. ; The Fundamental Region For a Fuchsian Group, {\it Bull. Amer. Math. Soc.}, vol. 31(1925), 531-539

\noindent [Ht] \textsc{Hutchinson} , J. I. ; A Method for Constructing the
Fundamental Region of a Discontinuous Group of Linear
Transformations, {\it Trans. Amer. Math. Soc.}, vol. 8(1907), 261-269.

\noindent [J1] \textsc{Johansson}, S. ; Genera  of Arithmetic Fuchsian Groups, {\it Acta Arith.}, vol. 86 (1998), no. 2, 171-191.

\noindent [J2] ----------- ; On Fundamental Domains of Arithmetic Fuchsian Groups, {\it Mathematics of Computation}, vol. 69, no. 229, 339-349.

\noindent [J3] ----------- ; A Description of Quaternion Algebras, in: {\it Arithmetic and Geometry of Quaternion Orders and Fuchsian Groups}, Doctoral Thesis, G\"{o}tborg Univ., 1997.

\noindent [M-R] \textsc{Maclachlan}, C. and \textsc{Reid}, A. ; {\it The Arithmetic of Hyperbolic 3-Manifolds}, Springer-Verlag, 2003.

\noindent [V] \textsc{Vign\'eras}, M-F ; {\it Arithm\'etique des Alg\`ebres de
Quaternions}, Springer LNM No. 800, Berlin (1980).\\\\

School of Mathematics, Institute for Research in Fundamental Sciences (IPM), P.O.Box 19395-5746, Tehran, Iran.

{\it Email:} \verb"jahangiri"@ \verb"mail.ipm.ir"

 \hskip 1.168 cm \verb"jahangiri"@ \verb"gmail.com"

\end{document}